\documentclass[12pt,twoside]{article}

\usepackage{a4}

\usepackage{amssymb,amsmath,amsthm,latexsym}
\usepackage{amsfonts}
\usepackage{amsfonts}
\usepackage{graphicx}

\numberwithin{subcase}{case}

\usepackage{amsmath, amsfonts}
\usepackage{amssymb, graphicx}
\usepackage{amscd}
\usepackage{textcomp}
\usepackage{palatino}
\usepackage{xcolor}
\usepackage{array}
\usepackage[colorlinks=true,linkcolor=red,citecolor=red]{hyperref}
\allowdisplaybreaks

\newtheorem{theorem}{Theorem}[section]

\newtheorem{corollary}[theorem] {Corollary}
\newtheorem{definition}[theorem]{Definition}

\newtheorem{lemma} [theorem]{Lemma}

\newtheorem{remark}[theorem]{Remark}

\newtheoremstyle{case}{}{}{}{}{}{:}{ }{}
\theoremstyle{case}

\usepackage{indentfirst}

\vspace{5cm}

\vspace{5cm}

\begin{document}
	\label{'ubf'}  
	\setcounter{page}{1} 
	\markboth
	{\hspace*{-9mm} \centerline{\footnotesize
			A note on necessary conditions for a friend of 10}}
	{\centerline{\footnotesize 
			T. Chatterjee, S. Mandal and S. Mandal
		} \hspace*{-9mm}}
	
	\vspace*{-2cm}
	\begin{center}
		
		{\textbf{A note on necessary conditions for a friend of 10}\\
			\vspace{.2cm}
			\medskip
			{\sc Tapas Chatterjee}\\
			{\footnotesize  Department of Mathematics,}\\
			{\footnotesize Indian Institute of Technology Ropar, Punjab, India.}\\
			{\footnotesize e-mail: {\it tapasc@iitrpr.ac.in}}
			\medskip
		
			{\sc Sagar Mandal}\\
			{\footnotesize  Department of Mathematics and Statistics,}\\
			{\footnotesize Indian Institute of Technology Kanpur, Uttar Pradesh, India.}\\
			{\footnotesize e-mail: {\it sagarm23@iitk.ac.in}}
			\medskip
            
			{\sc Sourav Mandal}\\
			{\footnotesize Department of Mathematics, }\\
			{\footnotesize RKMVERI, Belur Math, Howrah, West Bengal -711202, India.}\\
			{\footnotesize e-mail: {\it souravmandal1729@gmail.com}}
			\medskip}
		
	\end{center}
	\thispagestyle{empty} 
	\vspace{.5cm}
	
	\hrulefill
	\begin{abstract}  
		Solitary numbers are shrouded with mystery.  A folklore conjecture assert that 10 is a solitary number i.e. it has no friends. In this article,  we establish that if $N$ is a friend of $10$ then it must be odd square with at least seven distinct prime factors, with $5$ being the least one. Moreover there exists a prime factor $p$ of $N$ such that $2a+1\equiv 0 \pmod f$ and $5^{f}\equiv 1 \pmod p$ where $f$ is the smallest odd  positive integer greater than $1$ and less than or equal to $\min\{ 2a+1,p-1\}$, provided $5^{2a}\mid \mid N$. Further, there exist prime factors $p$ and $q$ (not necessarily distinct) of $N$ such that $p\equiv1 \pmod {10}$ and $q\equiv 1\pmod 6$. Besides, we prove that if a Fermat prime $F_k$ divides $N$ then $N$ must have a prime factor congruent to $1$ modulo $2F_k$. Also, if we consider the form of $N$ as $N=5^{2a}m^2$ then $m$ is non square-free. Furthermore, we show that  $\Omega(N)\geq 2\omega(N)+6a-4$ and if $\Omega(m)\leq K$ then $N< 5\cdot 6^{(2^{K-2a+1}-1)^2}$ where $\Omega(n)$ and $\omega(n)$ denote the total number of prime factors and the number of distinct prime factors of the integer $n$ respectively.
	\end{abstract}
	\noindent 
	{\textbf{Key words and phrases}:  Abundancy Index, Friendly Numbers, Solitary Numbers, Sum of Divisors.\\
		
		
		\vspace{.5cm}

\section{Introduction}
A positive integer $N>10$ is said to be a friend of 10 if $I(N)=I(10)=9/5$, where $I(N)$ is the abundancy index of $N$ which is defined as $I(N)=\sigma(N)/{N}$, where $\sigma(N)$ is the sum of positive divisors of $N$. Abundancy of any integer itself is an active area of research and attracted many mathematicians \cite{CMM, wp, rl}. In order to find one friend of 10, many authors gave necessary conditions for the existence of $N$. Unfortunately, we do not know a single friend so far. Lately, J. Ward proved that 
 \begin{theorem}[\cite{ward2008does}]
If $n$ is a friend of 10, then $n$ is a square with at least 6 distinct prime factors, the smallest being 5. Further, at least one of $n$’s prime factors must be congruent to 1 modulo 3, and appear in the prime power factorization of $n$ to a power congruent to 2 modulo 6. If there is only one such prime dividing $n$, then it appears to a power congruent to 8 modulo 18 in the factorization of
$n$.
\end{theorem}
In this paper, we improve the above result by the following theorem:

\begin{theorem}\label{thm 1.2}
    If $N$ is a friend of $10$, then $N$ it has at least $7$ distinct prime factors. Further, there exists prime factor $p$ and $q$ not necessarily distinct of $N$ such that $p\equiv1 \pmod {10}$ and $q\equiv 1\pmod 6$.
\end{theorem}

The proof of the above theorem requires many technical theorems and lemmas. We begin with those results and subsequently reach our main goal.

 \begin{theorem}\label{thm 1.3}
 Let $p,q$ be two distinct prime numbers with  $p^{k-1}\mid\mid (q-1)$ where $k$ is some positive integer. Then $p$ divides $\sigma(q^{2a})$ if and only if $2a+1 \equiv\ 0 \pmod f$ where $f$ is the smallest odd positive integer greater than $1$ such that $q^{f}\equiv 1 \pmod {p^{k}}$. 
\end{theorem}
We use the notation $f_{P}^{Q}$ to denote $f$ which depends on the prime numbers $p=P$ and $q=Q$ in Theorem \ref{thm 1.3}.
\begin{corollary}\label{coro 1.4}
Let $p$, $p^*$ and $q$ be three distinct prime numbers with $p^{k-1}\mid \mid (q-1)$ and $p^{*k^*-1}\mid \mid (q-1)$ $(k,k^*\in \mathbb{Z^{+}}$) also let $p \mid \sigma(q^{2a})$ and $p^* \mid \sigma(q^{2a^*})$( a,$a^*$$\in \mathbb{Z^{+}}$) with $f_{p}^{q}$, $f_{p^*}^{q}$ respectively in Theorem \ref{thm 1.3}. If $f_{p^*}^{q}\mid f_{p}^{q}$ then $pp^*\mid \sigma(q^{2a})$.  
\end{corollary}

\begin{corollary}\label{coro 1.5}
Let $N$ be a friend of $10$ with $5^{2a}\mid \mid N$, $a\in \mathbb{Z^+}$. Then there exists a prime factor $p$ of $N$ such that $2a+1\equiv 0 \pmod f$ and $5^{f}\equiv 1 \pmod p$ where $f$ is the smallest odd  positive integer greater than $1$ and less than or equal to $\min\{ 2a+1,p-1\}$.
\end{corollary}
\begin{corollary}\label{coro 1.6}
  Let $N$ be a friend of $10$. If any Fermat prime $F_{k}=2^{2^{k}}+1$ ($k\in \mathbb{Z}_{\geq0}$) divides $N$ then there exists a prime factor $p$ of $N$ such that $p\equiv 1 \pmod {2F_{k}}$.
\end{corollary}
\begin{theorem}\label{thm 1.7}
If $p$ and $q$ are two distinct prime numbers with $q>p$, then $p$ divides $\sigma(q^{2a}),~a \in\mathbb{Z^{+}}$ if and only if for $r\neq 1$, $q\equiv r \pmod p$ and $2a+1 \equiv 0 \pmod f$ where $f$ is the smallest odd positive integer  greater than $1$ such that $r^f\equiv 1 \pmod p$ and for $r=1$, $q\equiv r \pmod p$ and $2a+1\equiv 0 \pmod p$.
\end{theorem}
\begin{corollary}\label{coro 1.8}
If $N$ is a friend of $10$ then $N$ has a prime factor $p$ such that $p\equiv1 \pmod6$.
\end{corollary}

\begin{theorem}\label{thm 1.9}
    If $N=5^{2a}m^2$ is a friend of $10$ where $a,m \in \mathbb{N}$ then $m$ is non squarefree.
\end{theorem}
  
 The following theorem helps us to count the total number of prime divisors of any friend of 10. Here is the statement of the theorem. 
  
  \begin{theorem}\label{thm 1.10}
    If $N=5^{2a}m^2$ is a friend of 10 where $a,m \in \mathbb{N}$ then $\Omega(N)\geq 2\omega(N)+6a-4$.
  \end{theorem}
 
 As an immediate corollary, we get the following upper-bound of $N$. 
  
\begin{corollary}\label{coro 1.11}
     If $N=5^{2a}m^2$ is a friend of 10 where $a,m \in \mathbb{N}$ and $\Omega(m)\leq K$ then $N< 5\cdot 6^{(2^{K-2a+1}-1)^2}$.
\end{corollary}

  Before going into the proves, it is convenient to establish some notations and useful lemmas.

      \section{Preliminaries}
      In this section, we note some lemmas that will play a significant role in proving our main results.

\textbf{Properties of the Abundancy Index }\cite{rl,paw}.
 \begin{enumerate}
    \item $I(n)$ is weakly multiplicative, that is, if $n$ and $m$ are two coprime positive integers then $I(nm)=I(n)I(m)$.
    \item\label{p2} Let $a,n$ be two positive integers and $a>1$. Then $I(an)>I(n)$.
\item Let $p_1$, $p_2$, $p_3$,..., $p_m$ be $m$ distinct primes and $a_1$, $a_2$, $a_3$ ,..., $a_m$ be positive integers then
\begin{align*}
    I\biggl (\prod_{i=1}^{m}p_i^{a_i}\biggl)=\prod_{i=1}^{m}\biggl(\sum_{j=0}^{a_i}p_i^{-j}\biggl)=\prod_{i=1}^{m}\frac{p_i^{a_i+1}-1}{p_i^{a_i}(p_i-1)}.
\end{align*}
\item\label{p4} If $p_{1}$,...,$p_{m}$ are distinct primes, $q_{1}$,...,$q_{m}$ are distinct primes such that  $p_{i}\leq q_{i}$ for all $1\leq i\leq m$. If $t_1,t_2,...,t_m$ are positive integers then
\begin{align*}
   I \biggl(\prod_{i=1}^{m}p_i^{t_i}\biggl)\geq I\biggl(\prod_{i=1}^{m}q_i^{t_i}\biggl).
\end{align*}
\item\label{p5}  If $n=\prod_{i=1}^{m}p_i^{a_i}$, then $I(n)<\prod_{i=1}^{m}\frac{p_i}{p_i-1}$.
\end{enumerate}

      \textbf{Definition and Notation}\\

 Let $q$ be an odd prime. Define
    \begin{align*}
        \mathcal{E}_{q}(x)=\{p : \text{$p$ odd prime}, p \mid x, q\mid \sigma(p^\eta), \eta \geq 2, \eta~ \text{even}\}
        \end{align*}
    and  $|\mathcal{E}_q(N)|$ denotes the cardinality of $\mathcal{E}_q(N)$.\\\\
    Let  $$\mathcal{A}_{n, a}(r):=\left \{\sum_{i=1}^{r}a^{c_i}-r :\sum_{i=1}^{r} c_i =n, c_i \in \mathbb{N} \right \}$$ and $$\mathcal{L}_{n,a}:=\bigcup_{r=1}^{n}\mathcal{A}_{n,a}(r).$$
   
   \textbf{Notation:}
    \begin{itemize}
   \item $v_p(N)$ is denoted for $p^{v_{p}(N)}~\|~N$.
        \item $o_q(p)$ is denoted for the multiplicative order of p modulo q.
    \end{itemize}

\begin{lemma}\label{lem1}
    Let $p$ and $q$ be primes, $q \geq 3$, and $a \in \mathbb{N}$ then
\begin{align*}
    v_{q}(\sigma(p^a))=  \left\{
\begin{array}{ll}
      v_{q}(p^{o_q(p)}-1)+v_{q}(a+1) & \text{if~} o_{q}(p) \mid (a+1) \text{~and~} o_{q}(p) \neq 1, \\
      v_{q}(a+1) & \text{if~} o_{q}(p)=1\\
      0 & \text{otherwise.} \\
\end{array} 
\right.
\end{align*}
\end{lemma}
\begin{proof}
See \cite{John,nielsen2007odd} for its proof. 
\end{proof}

\begin{lemma}\label{lem3.3}
    Let a function $\psi: [1,\infty)	\rightarrow\mathbb{R}$ defined by $\psi(x)=ax-a^x$ then $\psi$ is a strictly decreasing function of $x$ in $[1,\infty)$ for all $a>e$.
\end{lemma}
\begin{proof}
     Let $\psi: [1,\infty)	\rightarrow\mathbb{R}$ is defined by $\psi(x)=ax-a^x$, then clearly
     \begin{align*}
         \psi'(x)=a-a^{x} \log a < 0; ~~~~\forall x \in [1,\infty)~ \text{and}~ \forall a > e.
            \end{align*}
            Hence, $\psi$ is a strictly decreasing function of $x$ in $[1,\infty)$ for all $a>e$.
\end{proof}
\begin{lemma}\label{lem3.4}
    Let $(c_{1}, c_{2}, ... , c_{k})
    $ be any partition of $n$ i.e; $\sum_{i=1}^{k} c_{i}=n$ for $1 \leq k \leq n$ and all $c_{i} \in \mathbb{N}$, then for any integer $a>e$ we have $$an < \sum_{i=1}^{k} a^{c_i}.$$
  
\end{lemma}
\begin{proof}
By lemma \ref{lem3.3}, for each $c_{i} \geq 1$ and for each $a > e$, we have
\begin{align}\label{3.3}
ac_{i}<a^{c_i}.
\end{align}
Since $(c_{1}, c_{2}, ... , c_{k})
    $ is a partition of $n$, we have 
    \begin{align*}
    a\sum_{i=1}^{k}c_{i}=an.
    \end{align*}
    Now from (\ref{3.3}), we have
    \begin{align*}
    a\sum_{i=1}^{k}c_{i}=an < \sum_{i=1}^{k}a^{c_{i}}.
    \end{align*}
    This completes the prove.
\end{proof}
\begin{definition}
An odd number $M$ is said to be an odd $m/d$-perfect number if \begin{align*}
\frac{\sigma(M)}{M}=\frac{m}{d}.
\end{align*}
\end{definition}
  \begin{lemma}\label{pn}
  If M is an odd m/d-perfect number with k distinct prime factors then $M < d(d+1)^{(2^k-1)^2}$.
  \end{lemma}
  \begin{proof}
  See \cite{pn} for a proof.
  \end{proof}
\section{Proofs of the Main Theorems}
\subsection{Proof of Theorem \ref{thm 1.3}.} 
Let $p$ and $q$ be two distinct primes with $p^{k-1} \mid \mid (q-1)$ where $k$ is some positive integer. Suppose that, $p$ divides $\sigma(q^{2a})$ where $a\in \mathbb{Z^+}$. Then

$$\sigma(q^{2a})\equiv 0 \pmod p$$
i.e.,
$$\sum_{r=0}^{2a}q^{r}\equiv 0 \pmod p$$
which is equivalent to,
$$\frac{q^{2a+1}-1}{q-1}\equiv 0 \pmod p.$$
Therefore,
$$q^{2a+1}\equiv 1 \pmod {p^{k}}.$$
By $Euler-Fermat$ theorem we can write
$$q^{\phi(p^{k})}\equiv\ 1 \pmod  {p^{k}}.$$
Let $f$ be the smallest positive integer such that $q^{f} \equiv 1 \pmod {p^{k}}$, then $f>1$ as if $f=1$ then $q \equiv 1 \pmod {p^{k}}$ i.e., $p^{k}\mid(q-1)$ which is impossible as $p^{k-1}\mid \mid (q-1)$. Since $2a+1\geq f$ we can write $2a+1=tf+r$ for some $t\in \mathbb{Z^+}$ and $0\leq r< f$. Then
$$1\equiv q^{2a+1}\equiv q^{tf+r}\equiv q^r \pmod {p^{k}}.$$
So $r=0$ and therefore we have $2a+1 \equiv 0 \pmod f$, by the same argument we can prove that $\phi(p^{k})\equiv\ 0 \pmod f$. Moreover $f$ must be an odd positive integer otherwise $2a+1 \not \equiv 0 \pmod f$.\\
Conversely, let $2a+1\equiv 0 \pmod  f$ where $f$ is the smallest odd positive integer greater than $1$ such that $q^{f}\equiv 1 \pmod {p^{k}}$. Then $2a+1=tf$, for some $t\in \mathbb{Z^{+}}$ and therefore 
$$q^{2a+1}=q^{tf} \equiv 1 \pmod {p^{k}}.$$
Since $(q-1,p^{k})=p^{k-1}$, we have

\begin{align*}
    \frac{q^{2a+1}-1}{(q-1)}\equiv 0 \pmod p
   \end{align*}
implies
\begin{align*}
\sigma(q^{2a})\equiv 0 \pmod p.
\end{align*}
In other words, $p$ divides $\sigma(q^{2a})$ and this proves the theorem.\qed\\
\begin{remark}\label{re 3.1}
  If $N=\prod_{i=1}^{m}p_i^{2a_i}$$(p_1=5)$ is a friend of 10 then $I(N)=\frac{9}{5}$ i.e., $\sigma(5^{2a_1})\cdot \prod_{i=2}^{m}\sigma(p_i^{2a_i})=9\cdot 5^{{2a_1}-1}\cdot \prod_{i=2}^{m}p_{i}^{2a_i}$. This shows that for every prime factors $p_i$ of $N$, there exists some $p_j$ other than $p_i$ such that $p_i\mid \sigma(p_j^{2a_j})$ and also for every $p_i$, the prime factors of $\sigma(p_i^{2a_i})$ are precisely belong to the set $\{3,5,p_{2},p_{3},...,p_{m}\}.$
\end{remark}
\subsection{Proof of Corollary \ref{coro 1.4}.}
Let $p$, $p^*$ and $q$ be three distinct primes and $p \mid \sigma(q^{2a})$ and $p^* \mid \sigma(q^{2a^*})$ with $f_{p}^{q}$, $f_{p^*}^{q}$ respectively in Theorem \ref{thm 1.3}. Also suppose that $f_{p^*}^{q}\mid f_{p}^{q}$. Then by Theorem \ref{thm 1.3} we have $2a+1\equiv 0 \pmod {f_{p}^{q}}$ and $2a^*+1\equiv 0 \pmod {f_{p^*}^{q}}$ and since $f_{p^*}^{q}\mid f_{p}^{q}$ we can write $2a+1\equiv 0 \pmod {f_{p^*}^{q}}$ and this implies that  $p^*\mid \sigma(q^{2a})$ by Theorem \ref{thm 1.3} and since $(p,p^*)=1$, we have $pp^*\mid \sigma(q^{2a})$. It completes the proof.\qed\\
\subsection{Proof of Corollary \ref{coro 1.5}.}
Let $N=\prod_{i=1}^{m}p_i^{2a_i}$$(p_1=5)$ with $5^{2a{_1}}\mid \mid N$ where $a{_1}\in \mathbb{Z^+}$, be a friend of $10$. In particular let $q=5$ in Theorem \ref{thm 1.3}. Then there must exists a prime $p$ of $N$ such that $p \mid \sigma(5^{2a_{1}})$. Since $p\geq5$, $p\nmid 4$, therefore we have $2a_{1}+1\equiv 0 \pmod f$ and $5^{f}\equiv 1 \pmod p$ where $f$ is the smallest odd positive integer greater than 1. As $f\mid (2a_{1}+1)$ and $f\mid \phi(p)$ i.e., $f\mid (p-1)$, we have $f\leq \min\{ 2a_{1}+1,p-1\}$. This completes the proof.\qed\\
\subsection{Proof of Corollary \ref{coro 1.6}.}
Let $N$ be a friend of $10$ and $F_{k}=2^{2^k}+1$ ($k\in \mathbb{Z}_{\geq0}$) be a prime factor of $N$.
If possible, assume that $N$ has no prime factor congruent to $1$ modulo $F_{k}$. Now since $F_{k}\mid N$, by Theorem \ref{thm 1.3} there exists a prime factor $p$ of $N$ with $p^{2a_p}\mid \mid N$ and $F_{k}^{k^*-1}\mid \mid (p-1)$ such that $2a_{p}+1\equiv 0 \pmod f$ and $p^{f}\equiv 1 \pmod {F_{k}^{k^*}}$ where $f$ is smallest positive odd integer greater than $1$. Since no prime factor $p$ of $N$ is congruent to $1$ modulo $F_{k}$, we have $k^*=1$. Further as $f\mid \phi(F_{k}^{k^*})$ i.e., $f \mid \phi(F_{k})$ and $\phi(F_{k})=2^{2^{k}}$, this shows that there does not exist any odd positive divisor of $\phi(F_{k})$ greater than $1$ as it has only one prime factor $2$. Therefore, by Theorem \ref{thm 1.3} we can conclude that $F_{k}$ does not divide $\sigma(p^{2a_p})$. Since $N$ is a friend of 10 and $p$ was taken arbitrary, using Remark \ref{re 3.1} we can conclude that $F_{k}$ cannot be a divisor of $N$ but this contradicts our assumption that $F_k$ is a prime factor of $N$. Hence there must exists a prime factor $p$ of $N$ such that $p\equiv 1 \pmod  {F_{k}}$. Since $(p-1)$ is even we have $p\equiv 1 \pmod  {2F_{k}}$.\qed\\
 \begin{remark}\label{re 3.2}
  If $N$ is a friendly number of $10$ then it has a prime factor $p$ such that $p\equiv 1 \pmod {10}$ as $5\mid N$ and $5$ is a Fermat prime. Also it is useful to note that as $17$ is a Fermat prime and if $17\mid N$ then $N$ must have a prime factor $p$, such that $p\equiv 1 \pmod {34}$.  
 \end{remark}

\subsection{Proof of Theorem \ref{thm 1.7}}
Let $p$ and $q$ be two distinct prime numbers ($q>p$) such that $p\mid \sigma(q^{2a})$ for some positive integer $a$. Then
\begin{align*}
    \sigma(q^{2a})\equiv 0 \pmod p
\end{align*}

i.e.,
\begin{center}
\begin{equation}\label{eq 1}
\sum_{j=0}^{2a}q^j\equiv 0 \pmod p.
\end{equation}
\end{center}
As $q>p$, we can write $q=pk+r$, for some $k\in \mathbb{Z^{+}}$ and $0\leq r\leq (p-1)$. If $r=0$ then $q=pk$ but $p \nmid \sigma((pk)^{2a})$, therefore $1\leq r\leq (p-1)$. From \eqref{eq 1} we have that
\begin{align*}
   \sum_{j=0}^{2a}(pk+r)^j\equiv 0 \pmod p 
\end{align*}

i.e.,
\begin{align*}
    \sum_{j=0}^{2a}r^j\equiv 0 \pmod p.
\end{align*}

If $r=1$ then
\begin{align*}
   2a+1 \equiv 0 \pmod p .
\end{align*}

And if $r\neq 1$ then
\begin{align*}
    \frac{r^{2a+1}-1}{r-1}\equiv 0 \pmod p.
\end{align*}

Since $(r-1,p)=1$ we have
$$r^{2a+1}\equiv 1 \pmod p .$$
By $Euler-Fermat$ theorem we can write
$$r^{p-1}\equiv 1 \pmod p .$$
Let $f$ be the smallest positive integer such that $r^f\equiv 1 \pmod p$. Then $f>1$ otherwise $r\equiv 1 \pmod p$ i.e., $r=1$. Since $2a+1\geq f$ we can write $2a+1=tf+s$ for some $t\in \mathbb{Z^{+}}$ and $0\leq s<f$. Then
\begin{align*}
  1\equiv r^{2a+1}\equiv r^{tf+s}\equiv r^s \pmod p . 
\end{align*}

So, $s=0$ and therefore we have $2a+1 \equiv 0 \pmod f$, by the same argument we can prove that $p-1 \equiv\ 0 \pmod f$. Moreover $f$ must be an odd positive integer otherwise $2a+1 \not \equiv 0 \pmod f$.\\
Conversely, let for $r\neq 1$, $q\equiv r \pmod p$ and $2a+1 \equiv 0 \pmod f$ where $f$ is the smallest odd positive integer greater than 1 such that $r^f\equiv 1 \pmod p$ and for $r=1$, $q\equiv r \pmod p$ and $2a+1\equiv 0 \pmod p$. \\
For $r\neq 1$, since $r^f\equiv 1 \pmod p$, we have
$q^f\equiv r^f\equiv 1 \pmod p$. If possible there exists $f'<f$ such that $q^{f'}\equiv 1 \pmod p$ then $1\equiv q^{f'}\equiv r^{f'} \pmod p$. This shows that $f'>f$ by our assumption on $f$. But this leads to a contradiction to that fact that $f'<f$, therefore $f$ is the smallest odd positive integer such that $q^f\equiv 1 \pmod p$ with $2a+1\equiv 0 \pmod f$. Since $p \nmid (q-1)$, by Theorem \ref{thm 1.3} we can conclude that $p\mid \sigma(q^{2a})$.\\
For $r=1$, we have $q\equiv 1 \pmod p$ and $2a+1\equiv 0 \pmod p$. For any positive integer $j$ we have

\begin{align*}
    q^{j}\equiv 1 \pmod p .
\end{align*}
Therefore,
\begin{align*}
   1+\sum_{j=1}^{2a}q^{j}\equiv 1+\sum_{j=1}^{2a}1 \pmod p 
\end{align*}
i.e.,
\begin{align*}
    \sum_{j=0}^{2a}q^{j}\equiv 2a+1 \pmod p
\end{align*}
which is equivalent to,
\begin{align*}
   \sigma(q^{2a})\equiv 0 \pmod p .
\end{align*}
In other words, $p$ divides $\sigma(q^{2a})$ and this completes the proof.\qed\\

\subsection{Proof of Corollary \ref{coro 1.8}.} Let $N$ be a friend of $10$ then we have $9\mid \sigma(N)$. This shows that, there exists a prime factor $p$ of $N$ with $p^{2a_p}\mid\mid N$ such that $3\mid \sigma(p^{2a_p})$. Since all prime factors of $N$ are greater than $3$, we can use Theorem \ref{thm 1.7}. Now we shall prove that $p$ can not be congruent to $2$ modulo $3$. Suppose that, $p\equiv 2 \pmod 3$ then we must have an odd positive integer $f>1$ such that $2^f\equiv 1 \pmod 3$ and $2a_p+1 \equiv 0 \pmod f$. But no such $f$ exists as for any odd positive integer $f$ we have $2^f\equiv 1 \pmod 3$ i.e., $(-1)^f\equiv 1 \pmod 3$ i.e., $-1\equiv 1 \pmod 3$, which is absurd. Therefore $p$ must be congruent to $1$ modulo $3$. Since $p-1$ is even we have $p\equiv 1 \pmod 6$.\qed\\\\

We use Corollary \ref{coro 1.4} to prove Theorem \ref{thm 1.2} in all cases. For example we write "if $11 \mid \sigma(5^{2a})$ then $71 \mid \sigma(5^{2a})$" as $f_{11}^5=5=f_{71}^5$ and $f_{11}^5 \mid f_{71}^5$ in that case we have $2a+1\equiv 0 \pmod 5$, again "if $19 \mid \sigma(5^{2a})$ then  $31 \mid \sigma(5^{2a})$" as $f_{19}^5=9$, $f_{31}^5=3$ and $f_{31}^{5} \mid f_{19}^{5}$ in that case we have $2a+1\equiv 0 \pmod 9$. The values of $f_{p}^{q}$ which we shall use in proving Theorem \ref{thm 1.2}, can be found after the conclusion part. 

\subsection{Proof of Theorem \ref{thm 1.2}.}
Finally, we are in a position to prove our main result. Here is the proof:

 Suppose that $N$ has exactly six distinct odd prime factors $p_1=5$, $p_2$, $p_3$, $p_4$, $p_5$ and $p_6$ with $p_6>p_5>p_4>p_3>p_2>p_1=5$. Then we can write $N=5^{2a_1}\cdot p_2^{2a_2}\cdot p_3^{2a_3}\cdot p_4^{2a_4}\cdot p_5^{2a_5}\cdot p_6^{2a_6}$ where $a_1$, $a_2$, $a_3$, $a_4$, $a_5$, $a_6$ are positive integers. Then the possible values of $p_i$, $2\leq i\leq 6$ are the following:

 \label{algebra}
  
  \begin{figure}[h]
    \centering
    \includegraphics[width=0.9\textwidth]{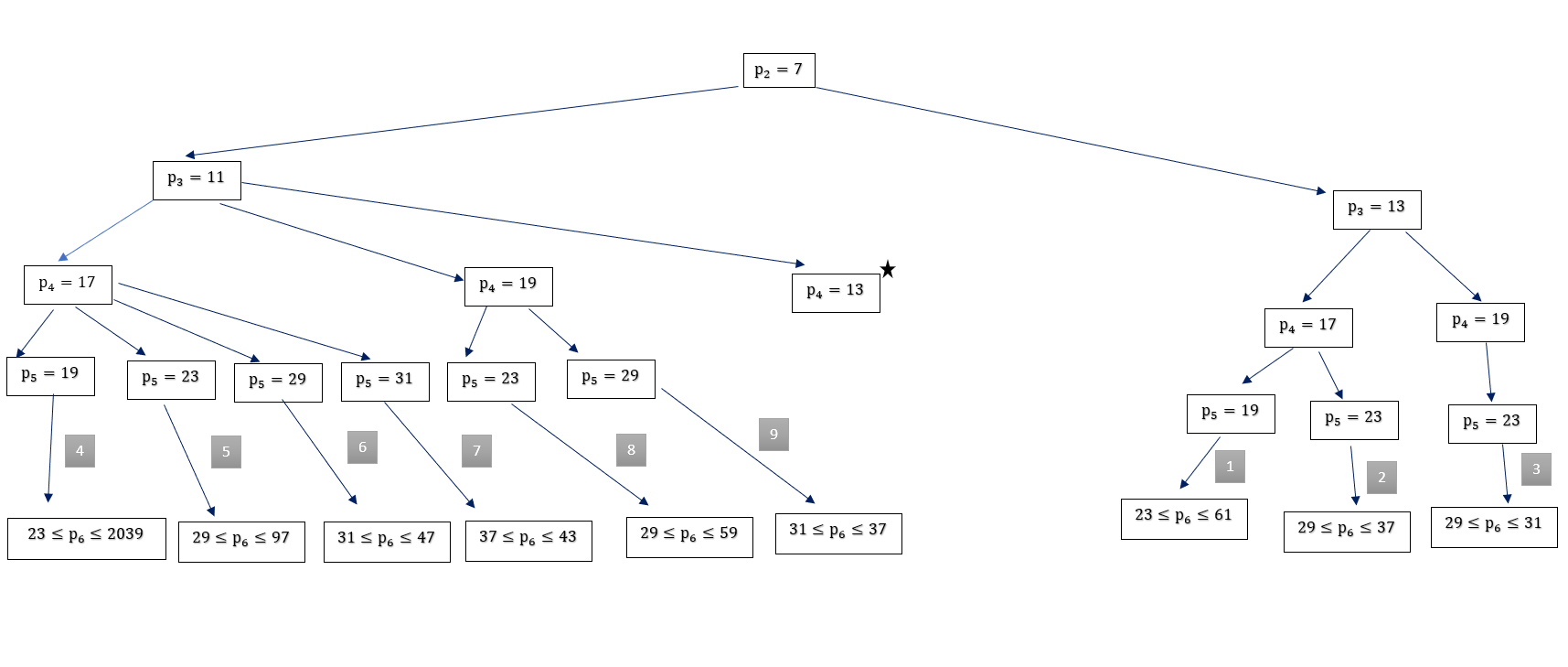}
    \caption{}
    \label{fig:enter-la}
\end{figure}
\label{algebra}
  \label{algebra}
  \begin{figure}[h]
    \centering
    \includegraphics[width=0.9\textwidth]{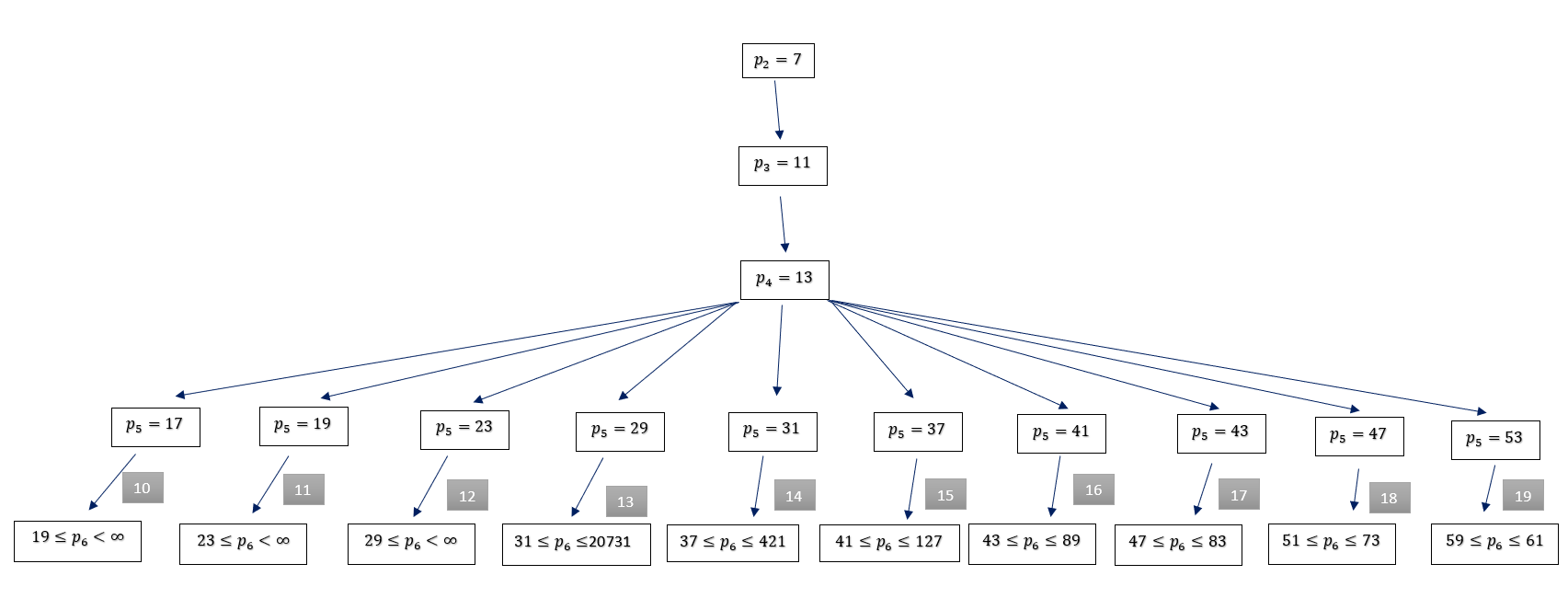}
    \caption{}
    \label{fig:enter-la}
\end{figure}
\label{algebra}
There are $19$ possible chain of values of $p_i$, $2\leq i\leq 6$ so that $I(N)=I(5^{2a_1}\cdot p_2^{2a_2}\cdot p_3^{2a_3}\cdot p_4^{2a_4}\cdot p_5^{2a_5}\cdot p_6^{2a_6})<\frac{9}{5}$ does not occur for all $2a_i$ i.e., for other combination of $p_i$ we always have $I(N)<\frac{9}{5}$. We shall begin with chain(1) and We shall show that the abundancy index of these possible chain of values of  $p_i$ even fails to give $\frac{9}{5}$.\\\\
\textbf{Case-1:}\\
For chain(1), $N=5^{2a_1}\cdot7^{2a_2}\cdot13^{2a_3}\cdot17^{2a_4}\cdot19^{2a_5}\cdot p_6^{2a_6}$, where $23\leq p_6\leq 61$ and $a_1$, $a_2$, $a_3$, $a_4$, $a_5$, $a_6$ are positive integers. But by Remark \ref{re 3.2} there must exists a prime factor of $N$ congruent to 1 modulo 34 as $17\mid N$ but, for any values of $p_6$, $N$ has no such prime factor. Therefore, $N$ cannot be a friend of $10$.\\\\
\textbf{Case-2:}\\
For chain(2), $N=5^{2a_1}\cdot7^{2a_2}\cdot13^{2a_3}\cdot17^{2a_4}\cdot23^{2a_5}\cdot p_6^{2a_6}$, where $29\leq p_6\leq 37$ and $a_1,a_2,a_3,a_4,a_5,a_6$ are positive integers. Using the same argument given in \textbf{Case-1}, we can conclude that $N$ cannot be a friend of $10$.\\\\
\textbf{Case-3:}\\
For chain(3), $N=5^{2a_1}\cdot7^{2a_2}\cdot13^{2a_3}\cdot19^{2a_4}\cdot23^{2a_5}\cdot p_6^{2a_6}$, where $29\leq p_6\leq 31$ and $a_1$, $a_2$, $a_3$, $a_4$, $a_5$, $a_6$ are positive integers. But $p_6=29$ is impossible as by Remark \ref{re 3.2}, there must exists a prime factor of $N$ congruent to $1$ modulo $10$. Therefore $p_6=31$ . Since $13\nmid (p_i-1)$ and $\phi(13)=12$, $f_{13}^{p_i}$ is $3$, for all $1\leq i\leq 6$. But $p_{i}^3\not\equiv 1 \pmod {13}$, for each $p_i$, $1\leq i\leq 6$. Therefore, by Theorem \ref{thm 1.3} we can see that, $13$ does not divide any $\sigma(p_i^{2a_i})$ and which is absurd by Remark \ref{re 3.1}. Therefore, $N$ cannot be a friend of $10$.\\\\
\textbf{Case-4:}\\
For chain(4), 
$N=5^{2a_1}\cdot7^{2a_2}\cdot11^{2a_3}\cdot17^{2a_4}\cdot19^{2a_5}\cdot p_6^{2a_6}$ where $23\leq p_6\leq 2039$ and $a_1$, $a_2$, $a_3$, $a_4$, $a_5$, $a_6$ are positive integers. Since $17\mid N$  by Remark \ref{re 3.2} there must exists a prime factor of $N$ congruent to $1$ modulo $34$, therefore $p_6$ must be congruent to $1$ modulo $34$. Now $5$ can divides $\sigma(11^{2a_3})$ or $\sigma(p_6^{2a_6})$. But if , $5\mid \sigma(11^{2a_3})$ then $3221\mid \sigma(11^{2a_3})$ that is $p_6$ must be $3221$. But $p_6=3221$ is not congruent $1$ modulo $34$. Hence 5 must divides $\sigma(p_6^{2a_6})$.  Since $9\mid \sigma(N)$, $3$ must divides one of $\sigma(p_i^{2a_i})$, for $1\leq i \leq 6$. But $3\nmid \sigma(p_i^{2a_i})$, for $i=1,3,4$. If $3\mid \sigma(19^{2a_5})$ then $127\mid \sigma(19^{2a_5})$ that is $p_6$ must be $127$ but again it is not congruent $1$ mod $34$. Now it is clear that $3\mid \sigma(7^{2a_2})$, we shall show that $9$ cannot divides $\sigma(7^{2a_2})$. If $9\mid \sigma(7^{2a_2})$ then this immediately implies  $37\mid \sigma(7^{2a_2})$ that is $p_6$ must be $37$ but it is not congruent $1$ modulo $34$. Therefore, $\sigma(p_6^{2a_6})$ must be divisible by $3$ as $9~|~\sigma(n)$. Hence $p_6$ must be congruent to $1$ modulo $34$ and $\sigma(p_6^{2a_6})$ must be divisible by $3$ and $5$. Using Theorem \ref{thm 1.7} we get, $p_6$ is congruent to $1$ modulo $510$. Therefore $p_6$ either $1021$ or $1531$. If $p_6=1021$ then $5\mid \sigma(1021^{2a_6})$ but this immediately allows $41\mid \sigma(1021^{2a_6})$ which is impossible as $p_i\neq 41$, for all $1\leq i\leq 6$. Hence the only choice for $p_6$ is $1531$ but again $5\mid \sigma(1531^{2a_6})$ implies $691\mid \sigma(1531^{2a_6})$ which is impossible by the same argument. Therefore there is no such $p_6$ for which $5$ can divides $\sigma(p_6^{2a_6})$. Hence $N$ cannot be a friend of $10$.\\\\
\textbf{Case-5:}\\
For chain(5), $N=5^{2a_1}\cdot7^{2a_2}\cdot11^{2a_3}\cdot17^{2a_4}\cdot23^{2a_5}\cdot p_6^{2a_6}$ where $29\leq p_6\leq 97$ and $a_1$, $a_2$, $a_3$, $a_4$, $a_5$, $a_6$ are positive integers. Using the same argument given in \textbf{Case-1}, we can conclude that $N$ cannot be a friend of $10$.\\\\
\textbf{Case-6:}\\
For chain(6), $N=5^{2a_1}\cdot7^{2a_2}\cdot11^{2a_3}\cdot17^{2a_4}\cdot29^{2a_5}\cdot p_6^{2a_6}$ where $31\leq p_6\leq 47$ and $a_1$, $a_2$, $a_3$, $a_4$, $a_5$, $a_6$ are positive integers. Using the same argument given in \textbf{Case-1}, we can conclude that $N$ cannot be a friend of $10$.\\\\
\textbf{Case-7:}\\
For chain(7), $N=5^{2a_1}\cdot7^{2a_2}\cdot11^{2a_3}\cdot17^{2a_4}\cdot31^{2a_5}\cdot p_6^{2a_6}$ where $37\leq p_6\leq 43$ and $a_1$, $a_2$, $a_3$, $a_4$, $a_5$, $a_6$ are positive integers. Using the same argument given in \textbf{Case-1}, we can conclude that $N$ cannot be a friend of $10$.\\\\
\textbf{Case-8:}\\
For chain(8), $N=5^{2a_1}\cdot7^{2a_2}\cdot11^{2a_3}\cdot19^{2a_4}\cdot23^{2a_5}\cdot p_6^{2a_6}$ where $29\leq p_6\leq 59$ and $a_1$, $a_2$, $a_3$, $a_4$, $a_5$, $a_6$ are positive integers. Now there are only three prime factors of $N$, $p_3=11$, $p_4=19$, $p_6=31$ or $p_6=59$ that can divide $\sigma(5^{2a_1})$. Now if $p_3=11\mid \sigma(5^{2a_1})$, immediately $71\mid \sigma(5^{2a_1})$ again if $p_4=19\mid \sigma(5^{2a_1})$, immediately $829\mid \sigma(5^{2a_1})$. Therefore $p_3=11$, $p_4=19$ cannot divide $\sigma(5^{2a_1})$ as for both the cases $p_6$ must be $71$ and $829$ respectively but $p_6\leq 59$. Therefore, $p_6$ must divides $\sigma(5^{2a_1})$ and $p_6$ must be either $31$ or $59$ as for other values of $p_6$, no $p_i$ will divides $\sigma(5^{2a_1})$ . If $p_6=31$ then $31\mid \sigma(5^{2a_1})$ which allows $829\mid \sigma(5^{2a_1})$ but it is impossible as $p_i\neq 829$, for all $1\leq i\leq 6$. Therefore only possible value of $p_6$ is $59$. But $35671\mid \sigma(5^{2a_1})$ whenever $59\mid \sigma(5^{2a_1})$, which is absurd as $p_i\neq 35671$, for all $1\leq i\leq 6$. Therefore, no prime factor of $N$ can divides $\sigma(5^{2a_1})$ that is $\sigma(5^{2a_1})$ has prime factors other than $3$, $p_1=5$, $p_2$,...,$p_5$, $p_6$. This proves that $N$ cannot be a friend of $10$.\\\\
\textbf{Case-9:}\\
For chain(9), $N=5^{2a_1}\cdot7^{2a_2}\cdot11^{2a_3}\cdot19^{2a_4}\cdot29^{2a_5}\cdot p_6^{2a_6}$ where $31\leq p_6\leq 37$ and $a_1$, $a_2$, $a_3$, $a_4$, $a_5$, $a_6$ are positive integers. Now there are only three prime factors of $N$, $p_3=11$, $p_4=19$, $p_6=31$ that can divide $\sigma(5^{2a_1})$. By the same argument given in \textbf{Case-8} we can conclude that  $N$ cannot be a friend of $10$.\\\\
\textbf{Case-10:}\\
For chain(10), $N=5^{2a_1}\cdot7^{2a_2}\cdot11^{2a_3}\cdot13^{2a_4}\cdot17^{2a_5}\cdot p_6^{2a_6}$ where $19\leq p_6<\infty$ and $a_1$, $a_2$, $a_3$, $a_4$, $a_5$, $a_6$ are positive integers. Then we have, $I(N)>I(5^{2a_1}\cdot7^{2a_2}\cdot11^{2a_3}\cdot13^{2a_4}\cdot17^{2a_5})\geq I(5^2\cdot7^2\cdot11^2\cdot13^2\cdot17^2)=\frac{31}{25}\cdot\frac{57}{49}\cdot\frac{133}{121}\cdot\frac{183}{169}\cdot\frac{307}{289}>\frac{9}{5}$. Therefore, $N$ cannot be a friend of $10$.\\\\
\textbf{Case-11:}\\
For chain(11),$N=5^{2a_1}\cdot7^{2a_2}\cdot11^{2a_3}\cdot13^{2a_4}\cdot19^{2a_5}\cdot p_6^{2a_6}$ where $23\leq p_6<\infty$ and $a_1$, $a_2$, $a_3$, $a_4$, $a_5$, $a_6$ are positive integers. Then we have
$I(N)>I(5^{2a_1}\cdot7^{2a_2}\cdot11^{2a_3}\cdot13^{2a_4}\cdot19^{2a_5})\geq I(5^2\cdot7^2\cdot11^2\cdot13^2\cdot19^2)=\frac{31}{25}\cdot\frac{57}{49}\cdot\frac{133}{121}\cdot\frac{183}{169}\cdot\frac{381}{361}>\frac{9}{5}$. Therefore, $N$ cannot be a friend of $10$.\\\\
\textbf{Case-12:}\\
For chain(12), $N=5^{2a_1}\cdot7^{2a_2}\cdot11^{2a_3}\cdot13^{2a_4}\cdot23^{2a_5}\cdot p_6^{2a_6}$ where $29\leq p_6<\infty$ and $a_1$, $a_2$, $a_3$, $a_4$, $a_5$, $a_6$ are positive integers. If $2a_1>2$ then , $I(N)>I(5^{2a_1}\cdot7^{2a_2}\cdot11^{2a_3}\cdot13^{2a_4}\cdot23^{2a_5})\geq I(5^4\cdot7^2\cdot11^2\cdot13^2\cdot23^2)=\frac{781}{625}\cdot\frac{57}{49}\cdot\frac{133}{121}\cdot\frac{183}{169}\cdot\frac{381}{361}>\frac{9}{5}$. Therefore, $2a_1\leq 2$ that is $2a_1=2$. Since $\sigma(5^{2a_1})=\sigma(5^2)=31$, $p_6$ must be $31$ but then $I(N)=I(5^{2a_1}\cdot7^{2a_2}\cdot11^{2a_3}\cdot13^{2a_4}\cdot23^{2a_5}\cdot31^{2a_6})\geq I(5^2\cdot7^2\cdot11^2\cdot13^2\cdot23^2\cdot31^2)=\frac{31}{25}\cdot\frac{57}{49}\cdot\frac{133}{121}\cdot\frac{183}{169}\cdot\frac{381}{361}\cdot\frac{993}{961}>\frac{9}{5}$. Therefore, $N$ cannot be a friend of $10$.\\\\
We shall discuss \textbf{Case-13} at the very end, as it is a little bit more tricky than others. Before that, we shall discuss the remaining cases. \\\\
\textbf{Case-14:}\\
For chain(14), $N=5^{2a_1}\cdot7^{2a_2}\cdot11^{2a_3}\cdot13^{2a_4}\cdot31^{2a_5}\cdot p_6^{2a_6}$ where $37\leq p_6\leq 421$ and $a_1$, $a_2$, $a_3$, $a_4$, $a_5$, $a_6$ are positive integers. Now $13\mid \sigma(N)$ if $p_6\in$\{53, 61, 79, 107, 113, 131, 139, 157, 191, 211, 263, 269, 313, 347, 367, 373, 419\} that is for other values of $p_6$, $13\nmid \sigma(N)$. If $p_6\in$ \{53,61,107,113,157,263,313,347,367,373\}, then $p_3=11$ and $p_5=31$ only can divide $\sigma(5^{2a_1})$. But if $11\mid \sigma(5^{2a_1})$ then $71\mid \sigma(5^{2a_1})$ and if $31\mid \sigma(5^{2a_1})$ then $829\mid \sigma(5^{2a_1})$. Therefore $p_3=11$, $p_5=31$ cannot divide $\sigma(5^{2a_1})$ as for both the cases $p_6$ must be $71$ and $829$ respectively but $p_6$ neither be $71$ nor be $829$. Therefore, $p_6$ must be one of $79,131,139,191,211,269,419$. If $p_6$ is one of them but does not divide $\sigma(5^{2a_1})$ then one of $11$, $31$ must divides $\sigma(5^{2a_1})$ but as previously discussed, it is not possible. Therefore, $p_6$ must divides $\sigma(5^{2a_1})$, but if $p_6=79\mid \sigma(5^{2a_1})$ then $305175781\mid \sigma(5^{2a_1})$, if $p_6=131$ or $211\mid \sigma(5^{2a_1})$ then $71\mid \sigma(5^{2a_1})$, if $p_6=139\mid \sigma(5^{2a_1})$ then $8971\mid \sigma(5^{2a_1})$, if $p_6=191\mid \sigma(5^{2a_1})$ then $6271\mid \sigma(5^{2a_1})$, if $p_6=269\mid \sigma(5^{2a_1})$ then $1609\mid \sigma(5^{2a_1})$, if $p_6=419\mid \sigma(5^{2a_1})$ then $191\mid \sigma(5^{2a_1})$. In each case, we are getting prime factors of $\sigma(5^{2a_1})$ other than $3$, $p_1=5$, $p_2$,..., $p_6$ and this is absurd by Remark \ref{re 3.1}. Therefore, no prime factor of $N$ can divides $\sigma(5^{2a_1})$ that is $\sigma(5^{2a_1})$ has prime factors other than $3$, $p_1=5$, $p_2$,..., $p_5,p_6$, which is absurd. This proves that $N$ cannot be a friend of $10$.\\\\
\textbf{Case-15:}\\
For chain(15), $N=5^{2a_1}\cdot7^{2a_2}\cdot11^{2a_3}\cdot13^{2a_4}\cdot37^{2a_5}\cdot p_6^{2a_6}$ where $41\leq p_6\leq 127$ and $a_1$, $a_2$, $a_3$, $a_4$, $a_5$, $a_6$ are positive integers. Now $13\mid \sigma(N)$ if $p_6\in$ \{53,61,79,107,113\} that is for other values of $p_6$, $13\nmid \sigma(N)$. If $p_6\in$ \{53,61,107,113\} we can use same argument given in \textbf{Case-8} to show that, $N$ cannot be a friend of $10$. Therefore we assume $p_6$ is $79$. Now if $p_6$ does not divide $\sigma(5^{2a_1})$ then $11\mid \sigma(5^{2a_1})$ which allows $71\mid \sigma(5^{2a_1})$ but this is absurd by the fact that $p_i\neq 71$, for all $1\leq i\leq 6$. Therefore, $p_6=79$ must divides $\sigma(5^{2a_1})$. But $31\mid \sigma(5^{2a_1})$ whenever $79\mid \sigma(5^{2a_1})$ and this shows that $p_6=79$ cannot divides $\sigma(5^{2a_1})$ as $p_i\neq 31$, for all $1\leq i\leq 6$. Therefore, no prime factor of $N$ can divides $\sigma(5^{2a_1})$ that is $\sigma(5^{2a_1})$ has prime factors other than $3$, $p_1=5$, $p_2$,...,$p_5$, $p_6$. This proves that $N$ cannot be a friend of $10$.\\\\
\textbf{Case-16:}\\
For chain(16), $N=5^{2a_1}\cdot7^{2a_2}\cdot11^{2a_3}\cdot13^{2a_4}\cdot41^{2a_5}\cdot p_6^{2a_6}$ where $43\leq p_6\leq 89$ and $a_1$, $a_2$, $a_3$, $a_4$, $a_5$, $a_6$ are positive integers. Now $13\mid \sigma(N)$ if $p_6\in$ \{53,61,79\} that is for other values of $p_6$, $13\nmid \sigma(N)$. If $p_6\in$ \{53,61\} we can use same argument given in \textbf{Case-8} to show that, $N$ cannot be a friend of $10$. Therefore, assume that $p_6=79$. In this case also we can use similar argument given in \textbf{Case-15} to show that, $N$ cannot be a friend of $10$.\\\\
\textbf{Case-17:}\\
For chain(17), $N=5^{2a_1}\cdot7^{2a_2}\cdot11^{2a_3}\cdot13^{2a_4}\cdot43^{2a_5}\cdot p_6^{2a_6}$ where $47\leq p_6\leq 83$ and $a_1$, $a_2$, $a_3$, $a_4$, $a_5$, $a_6$ are positive integers. Now $13\mid \sigma(n)$ if $p_6\in$ \{53,61,79\} that is for other values of $p_6$, $13\nmid \sigma(N)$. If $p_6\in$ \{53,61\} we can use same argument given in \textbf{Case-8} to show that, $N$ cannot be a friend of $10$. Therefore, assume that $p_6=79$. In this case also we can use similar argument given in \textbf{Case-15} to show that, $N$ cannot be a friend of $10$.\\\\
\textbf{Case-18:}\\
For chain(18), $N=5^{2a_1}\cdot7^{2a_2}\cdot11^{2a_3}\cdot13^{2a_4}\cdot47^{2a_5}\cdot p_6^{2a_6}$ where $51\leq p_6\leq 73$ and $a_1$, $a_2$, $a_3$, $a_4$, $a_5$, $a_6$ are positive integers. Now $13\mid \sigma(N)$ if $p_6=53$ or $p_6=61$ that is for other values of $p_6$, $13\nmid \sigma(N)$. For any $p_6$, only $11\mid \sigma(5^{2a_1})$ but it allows $71\mid \sigma(5^{2a_1})$ but this is absurd as, $p_i\neq 71$, for all $1\leq i\leq 6$. Therefore, $N$ cannot be a friend of $10$.\\\\
\textbf{Case-19:}\\
For chain(19), $N=5^{2a_1}\cdot7^{2a_2}\cdot11^{2a_3}\cdot13^{2a_4}\cdot53^{2a_5}\cdot p_6^{2a_6}$ where $59\leq p_6\leq 61$ and $a_1$, $a_2$, $a_3$, $a_4$, $a_5$, $a_6$ are positive integers. If $p_6$ is other than 59 then only $p_3=11$ can divides $\sigma(5^{2a_1})$, but if it does then $71\mid \sigma(5^{2a_1})$, which is absurd. Therefore assume that $p_6=59$. Now $p_3=11$, $p_6=59$ can divide $\sigma(5^{2a_1})$ but using similar argument given in \textbf{Case-8} we conclude that $N$ cannot be a friend of $10$.\\\\
\textbf{Case-13:}\\
For chain(13),  $N=5^{2a_1}\cdot7^{2a_2}\cdot11^{2a_3}\cdot13^{2a_4}\cdot29^{2a_5}\cdot p_6^{2a_6}$ where $31 \leq p_6 \leq 20731$  and $a_1$, $a_2$, $a_3$, $a_4$, $a_5$, $a_6$ are positive integers. It is clear that $3$ must divides one of $\sigma(p_i^{2a_i})$ for $1 \leq i \leq 6$ as $9\mid \sigma(n)$. But $3\nmid \sigma(p_i^{2a_i})$ for $i=1,3,5$. Therefore $3$ can divides $\sigma(7^{2a_2})$ or $\sigma(13^{2a_4})$or $\sigma(p_6^{2a_6})$. But if $3\mid \sigma(7^{2a_2})$ then $19\mid \sigma(7^{2a_2})$ but this is impossible as $p_6 \geq 31$ and $p_i\neq 19$, for $1\leq i\leq 5$. Now if $3\mid \sigma(13^{2a_4})$ then $61\mid \sigma(13^{2a_4})$ that is $p_6$ must be $61$. Then $11$ is the only prime factor of $N$ which can divides $\sigma(5^{2a_1})$, but $11\mid \sigma(5^{2a_1})$ implies $71\mid \sigma(5^{2a_1})$ which is impossible. Hence $3$ must divides $\sigma(p_6^{2a_6})$. Again $5$ must divides one of $\sigma(p_i^{2a_i})$ for $1 \leq i \leq 6$. But $5\nmid \sigma(p_i^{2a_i})$ for $i=1,2,4,5$. Therefore $5$ can divides $\sigma(11^{2a_3})$ and $\sigma(p_6^{2a_6})$. If $5\mid \sigma(11^{2a_3})$ then $3221\mid \sigma(11^{2a_3})$ which implies that $p_6=3221$ but this is absurd as $3\nmid \sigma(3221^{2a_6})$. Hence $5$ must divides $\sigma(p_6^{2a_6})$. Again $13$ must divides one of $\sigma(p_i^{2a_i})$ for $1 \leq i \leq 6$, but $13\nmid \sigma(p_i^{2a_i})$ for $i=1,2,3,4$. Therefore $13$ can divides $\sigma(29^{2a_4})$ and $\sigma(p_6^{2a_6})$. If $13\mid \sigma(29^{2a_5})$ then $67\mid \sigma(29^{2a_5})$ that is $p_6$ must be $67$ but this is impossible as $5\nmid \sigma(67^{2a_6})$. Therefore we have $3$,$5$,$13$ divide $\sigma(p_6^{2a_6})$. Since $3\mid \sigma(p_6^{2a_6})$ and $5\mid \sigma(p_6^{2a_6})$ we have $p_6\equiv 1 \pmod {30}$. Also we have $13\mid \sigma(p_6^{2a_6})$, and $p_6>13$ therefore by Theorem \ref{thm 1.7} we have $p_6\equiv 1 \pmod{13}$ or $p_6\equiv 3 \pmod{13}$ or $p_6\equiv 9 \pmod{13}$ after solving all conditions on $p_6$ we obtain, $p_6 \equiv 1 \pmod{390}$ or $p_6 \equiv 61 \pmod{390}$ or $p_6 \equiv 211 \pmod{390}$. If $p_6 \equiv 1 \pmod{390}$ and since $31 \leq p_6 \leq 20731$, $p_6 \in$ \{1171, 1951, 2341, 2731, 3121, 3511, 5851, 7411, 8191, 8581, 8971, 10141, 10531, 11311, 11701, 14431, 14821, 15601, 15991, 16381, 17551, 19501, 19891\}. TABLE 1 contains all such $p_6$ along with $p$ and $p^*$ such that whenever $p\mid \sigma(p_6^{2a_6})$ then $p^*\mid \sigma(p_6^{2a_6})$ also. If $p_6 \equiv 61 \pmod{390}$ and since $31 \leq p_6 \leq 20731$, $p_6 \in$ \{ 61, 1231, 1621, 2011, 2791, 3181, 3571, 5521, 6301, 6691, 8641, 9421, 9811, 12541, 13711, 15271, 15661, 16831, 20341, 20731 \}. TABLE 2 contains all such $p_6$ along with $p$ and $p^*$ such that whenever $p\mid \sigma(p_6^{2a_6})$ then $p^*\mid \sigma(p_6^{2a_6})$ also. If $p_6 \equiv 211 \pmod{390}$ and since $31 \leq p_6 \leq 20731$, $p_6 \in$ \{ 211, 601, 991, 1381, 2161, 2551, 3331, 4111, 5281, 6451, 6841, 7621, 8011, 9181, 11131, 12301, 14251, 15031, 16981, 17761, 18541, 20101 \}.TABLE 3 contains all such $p_6$ along with $p$ and $p^*$ such that whenever $p\mid \sigma(p_6^{2a_6})$ then $p^*\mid \sigma(p_6^{2a_6})$ also.
\clearpage
\begin{table}[h]
\begin{minipage}{.38\textwidth}
 
\begin{tabular}{ |c|c|c|  }

\hline
$p_6$ &$p$& $p^*$ \\
\hline
$1171$ &$3$&$65353$ \\
$1951$ &$3$&$79$ \\
$2341$ &$3$&$37$ \\
$2731$ &$3$&$61$ \\
$3121$ &$5$&$521$ \\
$3511$ &$5$&$61$ \\
$5851$ &$5$&$181$ \\
$7411$ &$5$&$31$ \\
$8191$ &$3$&$22366891$ \\
$8581$ &$3$&$31$ \\
$8971$ &$3$&$271$ \\
$10141$ &$3$&$43$ \\
$10531$ &$3$&$157$ \\
$11311$ &$3$&$37$ \\
$11701$ &$3$&$97$ \\
$14431$ &$3$&$157$ \\
$14821$ &$5$&$271$ \\
$15601$ &$5$&$31$ \\
$15991$ &$5$&$61$ \\
$16381$ &$3$&$1237$ \\
$17551$ &$3$&$31$ \\
$19501$ &$3$&$163$ \\
$19891$ &$3$&$331$ \\
\hline
\end{tabular}
   \caption{}
\label{tab:mytable}

\end{minipage}%
\begin{minipage}{.38\textwidth}
\begin{tabular}{ |c|c|c|  }

\hline
$p_6$ &$p$& $p^*$ \\
\hline
$61$ &$3$&$97$ \\
$1231$ &$3$&$37$ \\
$1621$ &$3$&$9631$ \\
$2011$ &$3$&$14821$ \\
$2791$ &$3$&$199807$ \\
$3181$ &$5$&$61$ \\
$3571$ &$5$&$101$ \\
$5521$ &$5$&$71$ \\
$6301$ &$5$&$31$ \\
$6691$ &$3$&$79$ \\
$8641$ &$5$&$3044081$ \\
$9421$ &$3$&$631$ \\
$12541$ &$5$&$151$ \\
$13711$ &$3$&$61$ \\
$15271$ &$3$&$43$ \\
$15661$ &$3$&$37$ \\
$16831$ &$3$&$109$ \\
$20341$ &$3$&$31$ \\
$20731$ &$5$&$251$ \\
\hline
\end{tabular}
   \caption{}
\label{tab:mytable}

\end{minipage}%
\begin{minipage}{.38\textwidth}
 
\begin{tabular}{ |c|c|c|  }

\hline
$p_6$ &$p$& $p^*$ \\
\hline
$211$ &$3$&$37$ \\
$601$ &$3$&$9277$ \\
$991$ &$3$&$277$ \\
$1381$ &$5$&$811$ \\
$2161$ &$3$&$119797$ \\
$2551$ &$3$&$79$ \\
$3331$ &$5$&$41$ \\
$4111$ &$5$&$491$ \\
$5281$ &$3$&$715237$ \\
$6451$ &$3$&$152461$ \\
$6841$ &$5$&$61$ \\
$7621$ &$3$&$223$ \\
$8011$ &$5$&$41$ \\
$9181$ &$3$&$31$ \\
$11131$ &$5$&$31$ \\
$12301$ &$3$&$31$ \\
$14251$ &$3$&$5207827$ \\
$15031$ &$3$&$199$ \\
$16981$ &$3$&$7394137$ \\
$17761$ &$3$&$379$ \\
$18541$ &$3$&$79$ \\
$20101$ &$3$&$37$ \\
\hline
\end{tabular}
   \caption{}
\label{tab:mytable}

\end{minipage}
\label{tab:tables}
\end{table}
This shows that for any $p_6$, we are getting some $p^{*}\not\in$ \{3, $p_1=5$, $p_2$,..., $p_6$\} such that $p^*\mid \sigma(p_6^{2a_6})$. In other words $\sigma(p_6^{2a_6})$ has other prime factors than $3$, $p_1=5$, $p_2$,..., $p_6$ which is absurd by Remark \ref{re 3.1}. Therefore $N$ cannot be a friend of $10$.
This proves that $N$ cannot have exactly $6$ distinct prime factors.

 The last part of the proof now follows from Remark \ref{re 3.2} and Corollary \ref{coro 1.8} .  \qed

\begin{remark}
In Theorem \ref{thm 1.7}, we set $p=5$ and $q > 5$ then as for $r= 2, 3, 4$ we have no odd $f >1$ such that $r^f \equiv 1 \pmod 5$ hence the only choice for $r$ is 1, thus we get $q \equiv 1 \pmod{10}$ with $2l+1 \equiv 0 \pmod{5}$.
\end{remark}
\begin{remark}\label{remark3.4}
 From  Lemma \ref{lem1} it is clear that if $q \equiv 1 \pmod {10}$ then $o_{5}(q)=1$ and if $v_{5}(\sigma(q^{2l}))=v_{5}(2l+1)=r,$ then $2l+1 \equiv 0 \pmod {5^r}$.
\end{remark}

\subsection{Proof of Theorem \ref{thm 1.9}.} 
Let $N=5^{2a}m^2$ be a friend of $10$ with $\omega(N)=s$, where $a,m \in \mathbb{N}$. Let us assume that $m$ is a square free integer. Then we can write $N=5^{2a}\prod_{j=1}^{s-1}p_{j}^2$ by letting $m=\prod_{j=1}^{s-1}p_{j}$, where $p_j$ are primes grater than 5. Then
\begin{align*}
    I(N)=\frac{\sigma(N)}{N}=\frac{9}{5}
\end{align*}
implies
\begin{align}\label{5}
  \sigma(5^{2a})\prod_{j=1}^ {s-1}\sigma(p_{j}^2)=9\cdot 5^{2a-1}\cdot \prod_{j=1}^ {s-1}p_{j}^2.
\end{align}
Since $a\in\mathbb{N}$, $5^{2a-1}\geq 5$. Again, note that $\gcd$(5, $\sigma(5^{2a})$)=1. Thus, we get $5 \mid \sigma(p_{t}^2)$ for some $1\leq t\leq s-1$ with $p_{t}\equiv 1 \pmod {10}$. This is impossible by the Remark \ref{remark3.4} as the exponent of $p_t$ is $2$. Hence, $5\nmid\sigma(5^{2a})\prod_{j=1}^{ s-1}\sigma(p_{j}^2)$, which  is  absurd. Therefore $m$ can not be a square-free integer.
   \begin{remark}
    From theorem \ref{thm 1.9}, m must have at least one prime factor which has an exponent greater than 2, which implies a trivial lower bound for $\Omega(N).$ i.e; $\Omega(N)\geq 2\omega(N)+2$.
\end{remark}
    
\begin{lemma}\label{o1}
The minimum element in the set $\mathcal{L}_{2a-1,5}$ is $8a-4$.
  
\end{lemma}
\begin{proof}
   If $\sum_{i=1}^{k}c_{i}=2a-1$, then from lemma \ref{lem3.4}, we get
   \begin{align}\label{0}
   5(2a-1)<\sum_{i=1}^{k}5^{c_i}.
   \end{align}
   Now since $k \leq 2a-1$, from (\ref{0}) we have
   \begin{align*}
   5(2a-1)-(2a-1)<\sum_{i=1}^{k}5^{c_i}-k.
   \end{align*}
   Thus, 
   
    \begin{align*}
   4(2a-1)<\sum_{i=1}^{k}5^{c_i}-k.
   \end{align*}
   This completes the prove.
\end{proof}

 \textbf{Note:}
     The minimum element in the set $\mathcal{L}_{2a-1,5}$ is from the set $\mathcal{A}_{2a-1,5}(2a-1).$   

    \subsection{Proof of Theorem \ref{thm 1.10}.}
    
 Let $N=5^{2a}\prod_{i=1}^{s-1}p_{i}^{2\gamma_{i}}$ be a friend of 10 with $\omega(N)=s$, then
 
 \begin{align*}
 \frac{\sigma(N)}{N}=\frac{9}{5}
 \end{align*}
 which implies,
 \begin{align}\label{o4.2}
 \sigma(5^{2a})\prod_{i=1}^{s-1}\sigma(p_{i}^{2\gamma_{i}})=9\cdot 5^{2a-1} \prod_{i=1}^{s-1}p_{i}^{2\gamma_{i}}
 \end{align}\\
 In the right hand side of (\ref{o4.2}), there are $2a-1$ numbers of 5. Now we consider
 the problem case by case.
 
\textbf{Case 1.} There is only one prime divisor of $N$ (without loss of generality, we may assume $p_{1}$) such that $$5^{2a-1}~\|~\sigma(p_{1}^{2\gamma_{1}})$$ i.e; $|\mathcal{E}_5(N)|=1$ and $p_1 \equiv 1 \pmod {10}$. Clearly, $v_{5}(\sigma(p_{1}^{2\gamma_{1}}))=2a-1$ where $2\gamma_{1}+1 \equiv 0 \pmod {5^{2a-1}}.$

   \textbf{Case 2.} There are two prime divisors of $N$ (without loss of generality, we may assume $p_{1}, p_{2}$) such that $$5^{2a-1}~\|~\sigma(p_{1}^{2\gamma_{1}})\sigma(p_{2}^{2\gamma_{2}})$$ i.e; $|\mathcal{E}_5(N)|=2$ and $p_1 \equiv p_2 \equiv 1 \pmod {10}$. Clearly, $v_{5}(\sigma(p_{1}^{2\gamma_{1}}))+ v_{5}(\sigma(p_{2}^{2\gamma_{2}}))=2a-1$, where $2\gamma_{i}+1 \equiv 0 \pmod {5^{v_{5}(\sigma(p_{i}^{2\gamma_{i}}))}}$ for $i=1, 2$ . 
   
   Continuing like this manner we have case $2a-2$.
   
\textbf{Case 2a-2.} There are $2a-2$ numbers of prime divisors of $N$ (without loss of generality, we may assume $p_{1}, p_{2}$, $p_{3}$,..., $p_{2a-2}$) such that $$5^{2a-1}~\|~\sigma(p_{1}^{2\gamma_{1}})\sigma(p_{2}^{2\gamma_{2}})\cdots \sigma(p_{2a-2}^{2\gamma_{2a-2}})$$ i.e; $|\mathcal{E}_5(N)|=2a-2$ and $p_1 \equiv p_2 \equiv \cdots \equiv p_{2a-2} \equiv 1 \pmod {10}$. Clearly, $\sum_{i=1}^{2a-2}v_{5}(\sigma(p_{i}^{2\gamma_{i}}))=2a-1$, where $2\gamma_{i}+1 \equiv 0 \pmod {5^{v_{5}(\sigma(p_{i}^{2\gamma_{i}}))}}$ for $i=1, 2,..., 2a-2 $ . In this case, all $v_{5}(\sigma(p_{k_i}^{2\gamma_{i}}))=1$ except one $v_{5}(\sigma(p_{j}^{2\gamma_{j}}))$ which is 2.

      \textbf{Case 2a-1.} There are exactly  $2a-1$ numbers of prime divisors of $N$ (without loss of generality, we may assume $p_{1}$, $p_{2}$, $p_{3}$,..., $p_{2a-1}$) such that $$5^{2a-1}~\|~\sigma(p_{1}^{2\gamma_{1}})\sigma(p_{2}^{2\gamma_{2}})\cdots \sigma(p_{2a-1}^{2\gamma_{2a-1}})$$ i.e; $|\mathcal{E}_5(N)|=2a-1$ and $p_1 \equiv p_2 \equiv \cdots \equiv p_{2a-1} \equiv 1 \pmod {10}$. Clearly, $\sum_{i=1}^{2a-1}v_{5}(\sigma(p_{i}^{2\gamma_{i}}))=2a-1$, where $2\gamma_{i}+1 \equiv 0 \pmod {5^{v_{5}(\sigma(p_{i}^{2\gamma_{i}}))}}$ for $i=1, 2,..., 2a-1$. In this case, all $v_{5}(\sigma(p_{i}^{2\gamma_{i}}))=1.$
      
If we consider all the cases mentioned above and choose the minimum value of $\sum_{i=1}^{t}2\gamma_{i}$ such that $5^{2a-1}~\|~\sigma(p_{1}^{2\gamma_{1}})\sigma(p_{2}^{2\gamma_{2}})\sigma(p_{3}^{2\gamma_{3}})\cdots \sigma(p_{t}^{2\gamma_{t}})$ satisfying $\sum_{i=1}^{t}v_{5}(\sigma(p_{i}^{2\gamma_{i}}))=2a-1$, for $1\leq t \leq |\mathcal{E}_5(N)|=2a-1$ and while excluding the primes with exponent 2 that are not included in those cases, we can obtain the minimum possible value for $\Omega(N)$.

To find the minimum value of $\sum_{i=1}^{t}2\gamma_{i}$, it is enough to consider  $2\gamma_j=5^{v_{5}(\sigma(p_{j}^{2\gamma_{j}}))}-1$  since $2 \gamma_{i}+1 \equiv 0 \pmod{v_{5}(\sigma(p_{i}^{2\gamma_{i}}))}$, for $j=1, 2,..., t$.

A careful calculation gives the minimum value as following: 
\begin{align}\label{6}
\sum_{i=1}^{t}2\gamma_{i}=\sum_{i=1}^{t}\bigg(5^{v_{5}(\sigma(p_{i}^{2\gamma_{i}}))}-1\bigg)=\sum_{i=1}^{t}5^{v_{5}(\sigma(p_{i}^{2\gamma_{i}}))}-t.
\end{align}
 Now, the sum in (\ref{6}) is same as the minimum element in the set $\mathcal{L}_{2a-1,5}$.
 
  From lemma \ref{o1} we know that the minimum element in the set $\mathcal{L}_{2a-1,5}$ is $8a-4$. Hence, the minimum value of $\Omega(N)$ is $2\omega(N)+6a-4$ i.e.; $\Omega(N) \geq 2\omega(N)+6a-4$. This completes the proof.
  
\begin{remark}
Since we have $\Omega(N) \geq 2\omega(N)+6a-4$, now using the  completely additive property of $\Omega(N)$ we have $\Omega(5^{2a})+\Omega(m^2)=2a+ 2 \Omega(m)\geq 2\omega(N)+6a-4$ i.e; 
\begin{align}\label{ap}
\Omega(m) \geq \omega(N)+ 2a-1.
\end{align}
Finally, using additive property of $\omega(N)$, we have from (\ref{ap})

\begin{align}
\Omega(m) \geq \omega(m)+ 2a
\end{align}
\end{remark}
\subsection{Proof of Corollary \ref{coro 1.11}.} 
 Since $\Omega(m) \leq K$, we have from (\ref{ap}) $K-2a+1 \geq \omega(N)$.
Since $N$ is an odd integer with abundancy index $9/5$ so, it also called an odd $9/5$-perfect number. Now using lemma \ref{pn}, we have 
$N< 5\cdot 6^{{(2^{{\omega(N)}}-1})^2}< 5\cdot 6^{(2^{K-2a+1}-1)^2}$.

\section{Conclusion}

We can of course use this method to check whether $n$ has seven or more distinct prime factors or not and we can predict the properties of prime factors of $n$ using the proven theorems. $10$ is not the only number whose classification is unknown to us. In fact the status of $14$, $15$, $20$ and many others are active topics for research. It may be proving whether $10$ has a friend or not is as much as difficult as finding an odd perfect number. However computer search shows that if $10$ has a friend then its smallest friend will be bigger than $10^{30}$\cite{OEIS}. The list of values of $f_{p}^{q}$ used in proving Theorem \ref{thm 1.2} can be found below.
\clearpage
\begin{center}
\begin{table}[h]
\begin{tabular}{|*{6}{c|}}
\hline
$f_{31}^{5}=3$ & $f_{11}^{5}=5$ & $f_{71}^{5}=5$ & $f_{19}^{5}=9$ & $f_{829}^{5}=9$ & $f_{305175781}^{5}=13$ \\
\hline
$f_{191}^{5}=19$ & $f_{6271}^{5}=19$ & $f_{19}^{5}=9$ & $f_{8971}^{5}=23$ & $f_{59}^{5}=29$ & $f_{35671}^{5}=29$\\
\hline
$f_{211}^{5}=35$ & $f_{79}^{5}=39$ & $f_{131}^{5}=65$ & $f_{269}^{5}=67$ &$f_{1609}^{5}=67$   & $f_{139}^{5}=69$  \\
\hline
$f_{419}^{5}=209$  & $f_{3}^{7}=3$  &$f_{19}^{7}=3$ & $f_{9}^{7}=9$ & $f_{37}^{7}=9$ & $f_{5}^{11}=5$ \\
\hline
$f_{3221}^{11}=5$ & $f_{3}^{13}=3$ & $f_{61}^{13}=3$ & $f_{3}^{19}=3$ & $f_{127}^{19}=3$ & $f_{13}^{29}=3$ \\
\hline
$f_{67}^{29}=3$ & $f_{3}^{61}=3$ & $f_{97}^{61}=3$ & $f_{3}^{211}=3$ & $f_{37}^{211}=3$ & $f_{3}^{601}=3$ \\
\hline
$f_{9277}^{601}=3$ & $f_{3}^{991}=3$ & $f_{277}^{991}=3$ & $f_{5}^{1021}=5$ &$f_{41}^{1021}=5$&$f_{1451}^{1021}=5$ \\
\hline
$f_{3}^{1171}=3$  & $f_{65353}^{1171}=3$ & $f_{3}^{1231}=3$ & $f_{37}^{1231}=3$  & $f_{5}^{1381}=5$  & $f_{811}^{1381}=5$ \\
\hline
$f_{5}^{1531}=5$ & $f_{691}^{1531}=5$ & $f_{3}^{1621}=3$ &
$f_{9631}^{1621}=3$ & $f_{3}^{1951}=3$ & $f_{79}^{1951}=3$ \\
\hline
$f_{3}^{2011}=3$ & $f_{14821}^{2011}=3$ & $f_{3}^{2161}=3$ & $f_{119797}^{2161}=3$ & $f_{3}^{2341}=3$ & $f_{37}^{2341}=3$ \\
\hline
$f_{3}^{2551}=3$ &$f_{79}^{2551}=3$ & $f_{3}^{2731}=3$ & $f_{61}^{2731}=3$ & $f_{3}^{2791}=3$ & $f_{199807}^{2791}=3$ \\
\hline
$f_{5}^{3121}=5$ & $f_{521}^{3121}=5$ & $f_{5}^{3181}=5$ & $f_{61}^{3181}=5$ & $f_{5}^{3331}=5$ &$f_{41}^{3331}=5$ \\
\hline
$f_{5}^{3511}=5$ & $f_{61}^{3511}=5$ & $f_{5}^{3571}=5$ & $f_{101}^{3571}=5$ & $f_{5}^{4111}=5$ & $f_{491}^{4111}=5$ \\
\hline
$f_{3}^{5281}=3$ & $f_{715237}^{5281}=3$ & $f_{5}^{5521}=5$& $f_{71}^{5521}=5$ & $f_{5}^{5851}=5$ & $f_{181}^{5851}=5$ \\
\hline
$f_{5}^{6301}=5$ & $f_{31}^{6301}=5$ & $f_{3}^{6451}=3$ & $f_{152461}^{6451}=3$ & $f_{3}^{6691}=3$ & $f_{79}^{6691}=3$ \\
\hline
$f_{5}^{6841}=5$ & $f_{61}^{6841}=5$ & $f_{5}^{7411}=5$ & $f_{31}^{7411}=5$ & $f_{3}^{7621}=3$ & $f_{223}^{7621}=3$ \\
\hline
$f_{5}^{8011}=5$ & $f_{41}^{8011}=5$ & $f_{3}^{8191}=3$ & $f_{22366891}^{8191}=3$ & $f_{3}^{8581}=3$& $f_{31}^{8581}=3$\\
\hline
$f_{5}^{8641}=5$& $f_{3044081}^{8641}=5$ &$f_{3}^{8971}=3$ &$f_{271}^{8971}=3$ & $f_{3}^{9181}=3$&$f_{31}^{9181}=3$ \\
\hline
$f_{3}^{9421}=3$&$f_{631}^{9421}=3$ &$f_{3}^{10141}=3$&$f_{43}^{10141}=3$& $f_{3}^{10531}=3$&$f_{157}^{10531}=3$ \\
\hline
$f_{5}^{11131}=5$ &$f_{31}^{11131}=5$ &$f_{3}^{11311}=3$ & $f_{37}^{11311}=3$& $f_{3}^{11701}=3$&$f_{97}^{11701}=3$ \\
\hline
$f_{3}^{12301}=3$ & $f_{31}^{12301}=3$& $f_{5}^{12541}=5$&$f_{151}^{12541}=5$ &$f_{3}^{13711}=3$ & $f_{61}^{13711}=3$\\
\hline
$f_{3}^{14251}=3$ &$f_{5207827}^{14251}=3$ & $f_{3}^{14431}=3$&$f_{157}^{14431}=3$ &$f_{5}^{14821}=5$ &$f_{271}^{14821}=5$\\
\hline
$f_{3}^{15031}=3$& $f_{199}^{15031}=3$& $f_{3}^{15271}=3$&$f_{43}^{15271}=3$ & $f_{5}^{15601}=5$&$f_{127}^{15601}=5$ \\
\hline
$f_{3}^{15661}=3$& $f_{37}^{15661}=3$& $f_{5}^{15991}=5$&
 $f_{61}^{15991}=5$ & $f_{3}^{16381}=3$&$f_{1237}^{16381}=3$\\
 \hline
 $f_{3}^{16831}=3$ & $f_{109}^{16831}=3$& $f_{3}^{16981}=3$&$f_{7394137}^{16981}=3$ &$f_{3}^{17551}=3$ &$f_{31}^{17551}=3$ \\
 \hline
 $f_{3}^{17761}=3$& $f_{379}^{17761}=3$ & $f_{3}^{18541}=3$ & $f_{79}^{18541}=3$ & $f_{3}^{19501}=3$ & $f_{163}^{19501}=3$\\
 \hline
  $f_{3}^{19891}=3$ & $f_{331}^{19891}=3$ & $f_{3}^{20101}=3$ & $f_{37}^{20101}=3$ & $f_{3}^{20341}=3$&$f_{31}^{20341}=3$\\
  \hline
  $f_{5}^{20731}=5$ & $f_{251}^{20731}=5$ & &&&\\

\hline
\end{tabular}
\caption{List of $f_{p}^{q}$}
\end{table}
\end{center}

\section{Data Availablity} 	
The authors confirm that their manuscript has no associated data.

\section{Competing Interests}
The authors confirm that they have no competing interest.

\end{document}